\documentclass[10pt]{amsart}
\usepackage[cp1251]{inputenc}
\usepackage[english,russian]{babel}
\usepackage{amsmath}
\usepackage{amssymb}
\usepackage{amsfonts}

\makeindex
\newtheorem{theorem}{Theorem}
\newtheorem{definition}{Definition}
\newtheorem{corollary}{Corollary}
\newtheorem{proposition}{Proposition}

\begin{document}
\begin{center}
{\Large  On Rota-Baxter operators of non-zero weight induced by  non
skew-symmetric solutions of the classical Yang-baxter equation  on
simple anticommutative  algebras.}
\end{center}
\author{{M.E. Goncharov}}
\email{goncharov.gme@gmail.com}

\vspace{7mm}

\begin{center}
{\bf M.\,E.\,Goncharov}
\end{center}
\renewcommand{\refname}{References}
\renewcommand{\proofname}{Proof.}

\renewcommand{\abstractname}{Abstract}
\begin{quote}
\noindent{\sc Abstract. }  Let $L$ be a simple anti-commutative
algebra.  In this paper we prove that  a non skew-symmetric solution
of the classical Yang-Baxter equation on $L$ with $L$-invariant
symmetric part induces on $L$ a Rota-Baxter operator of a non-zero
weight. \medskip

 \end{quote}
 \vspace{6pt}

\sloppy Given an algebra A over a field $F$ and scalar $\lambda \in
F$, a linear operator $R: A \rightarrow A$ is called a Rota—Baxter
of the weight $\lambda$ if for all $x,y\in A$ the following identity
holds:
\begin{equation}\label{el}
R(x)R(y) = R(R(x)y + xR(y) + \lambda xy).
\end{equation}

As an example of a Rota-Baxter operation of weight zero one can
consider the operation of integration on the algebra of continuous
functions on $\mathbb{R}$: the equation \eqref{el} follows from the
integration by parts formula.

An algebra with a Rota-Baxter operation is called a Rota-Baxter
algebra. These algebras first appeared in the paper of Baxter
\cite{Br}. The combinatorial properties of Rota-Baxter algebras and
operations were studied in papers of F.V. Atkinson, P. Cartier,
G.-C. Rota and the others (see \cite{Atk}-\cite{Car}). For basic
results and the main properties of Rota-Baxter algebras see
\cite{Guo}.

There is a standard method for constructing Rota-Baxter operations
of zero weight on a simple Lie algebra $L$ from skew-symmetric
solutions of the classical Yang-Baxter equations (CYBE): if $r=\sum
a_i\otimes b_i$ is a skew-symmetric solution of CYBE, then one can
define an operator $R$ on $L$ by $R(a)=\sum\langle a_i,a\rangle
b_i$, where $\langle \cdot, \cdot \rangle$ is a Killing form on L.
It turns out that $R$ is a Rota-Baxter operator of weight 0
(\cite{BD}-\cite{STS}).

In this paper we consider the case when $L$ is a simple
anti-commutative algebra and $r\in L\otimes L$ is a non
skew-symmetric solution of CYBE with $L$-invariant symmetric part.
It turns out that instead of the element $r$ it is more convenient
to consider structure that
 closely connected with the solutions of CYBE(skew-symmetric or
not): the structure of bialgebra (a vector space with multiplication
and comultiplication).

Lie bialgebras were introduced by Drinfeld \cite{Drinf} for studying
the solutions of the classical Yang-Baxter equation on Lie algebras.
Lie bialgebras are Lie algebras and Lie coalgebras at the same time,
such that comultiplication is a 1-cocycle. Every skew-symmetric
solution of the classical Yang-Baxter equation induces a structure
of a Lie bialgebra on the corresponding Lie algebra $L$. It is known
that in this case the Drinfeld double contain a non-zero radical.

If $r$ is not a skew-symmetric solution of CYBE then the
corresponding comultiplication gives a structure of a Lie bialgebra
if and only if $r+\tau(r)$ is a $L$-invariant element of $L\otimes
L$. Here $\tau$ is a switch morphism.

 In section 2 we consider  a structure of bialgebra on an arbitrary simple finite-dimensional algebra $A$ over a field of characteristic zero
with a semisimple Drinfeld double. We prove that the structure
induce on $A$  Rota-Baxter operators of a non-zero weight.

In section 3 we consider anti-commutative bialgebras with semisimple
non-simple Drinfeld double constructed on simple anti-commutative
algebras. We prove that in this case Rota-Baxter operators from
section 2 are induced by non skew-symmetric solutions of the
classical Yang-Baxter equations with $ad$-invariant symmetric part.
As a corollary we obtain that if $L$ is a simple Lie algebra and
$r=\sum a_i\otimes b_i$ is a non skew-symmetric solutions of the
classical Yang-Baxter equations such that $r+\tau(r)$ is a
$L$-invariant, then an operator $R$ on $L$ defined by
$R(a)=\sum\langle a_i,a\rangle b_i$ is a Rota-Baxter operator of a
non-zero weight.

In the last section we use the results from the section 3 and
construct Rota-Baxter operators of non-zero weight on simple non-Lie
Malcev algebra.

\section{Definitions and preliminary results.}

In the paper it is assumed that the characteristic of the ground
field $F$ is 0 and all spaces are supposed to be finite-dimensional.

 Given vector spaces $V$ and $U$ over a field $F$, denote
by $V\otimes U$ its tensor product over $F$. Define the linear
mapping $\tau$ on $V$ by $\tau(\sum\limits_ia_i\otimes
b_i)=\sum\limits_ib_i\otimes a_i$. Denote by $V^*$ the dual space of
$V$.

\begin{definition} A pair $(A, \Delta)$, where $A$ is a vector space over $F$ and
$\Delta : A \rightarrow A\otimes A$ is a linear mapping, is called a
coalgebra, while $\Delta$ is a comultiplication.
\end{definition}

The following definition  of a coalgebra related to some variety of
algebras was given in \cite{ANQ}.

{\bf Definition.} Let $\mathcal{M}$ be an arbitrary variety of
algebras. The pair $(A, \Delta)$ is called a $\mathcal{M}$-coalgebra
if $A^*$ belongs to $\mathcal{M}$.

Given $a\in A$, put $\Delta(a) = \sum a_{(1)}\otimes a_{(2)}$.

Define a multiplication on $A^*$ by
$$
 fg(a)=\sum f(a_{(1)})g(a_{(2)}),
$$
where $f,g\in A^*$, $a\in A$ and $\Delta(a)=\sum a_{(1)}\otimes
a_{(2)}$.  The algebra obtained is  \emph{the dual algebra} of the
coalgebra $(A, \Delta)$.

The dual algebra $A^*$ of $(A,\Delta)$ gives rise to the following
bimodule actions on $A$:
$$
f\rightharpoonup a=\sum a_{(1)}  f(a_{(2)})\text{ and
}a\leftharpoonup f=\sum f(a_{(1)}) a_{(2)},
$$
where $a\in A,\ f\in A^*$ and $\Delta(a)=\sum a_{(1)}\otimes
a_{(2)}$.

Let $A$ be an arbitrary algebra with a comultiplication $\Delta$,
and let $A^*$ be the dual algebra for $(A, \Delta)$. Then $A$
induces the bimodule action on $A^*$ by the formulas
$$
 f\leftharpoondown a(b)= f(ab)\text{ and }
 b\rightharpoondown f(a)= f(ab),
$$
where $a,b\in A,\ f\in A^*$.

Consider the space $D(A) = A \oplus A$ and equip it with a
multiplication by putting
$$
(a+f)(b+g)=(ab+f\rightharpoonup b+a\leftharpoonup
g)+(fg+f\leftharpoondown b+a\rightharpoondown g).
$$
Then $D(A)$ is an ordinary algebra over $F$, $A$ and $A^*$ are some
subalgebras in $D(A)$. It is called \emph{the Drinfeld double}.

Let $Q$ be a bilinear form on $D(A)$ defined by
$$
Q(a+f,b+g)= g(a)+ f(b)
$$
for all $a,b\in A$ and $f,g\in A^*$. It is easy to check that $Q$ is
a nondegenerate symmetric associative form, that is
$Q(xy,z)=Q(x,yz)$.

In \cite{Drinf} the following definition was given:

\begin{definition} Let $L$ be a Lie algebra with a comultiplication
$\Delta$. The pair $(L,\Delta)$ is called a Lie bialgebra if and
only if $(L,\Delta)$ is a Lie coalgebra and $\Delta$ is a 1-cocycle,
i.e., it satisfies
$$
\Delta([a,b])=\sum([a_{(1)},b]\otimes
a_{(2)}+a_{(1)}\otimes[a_{(2)},b])+\sum([a,b_{(1)}]\otimes
b_{(2)}+b_{(1)}\otimes[a,b_{(2)}])
$$
for all $a,b\in L$.
\end{definition}
In \cite{Drinf}, it was proved that the pair $(L,\Delta)$ is a Lie
bialgebra if and only if its Drinfeld double $D(L)$ is a Lie
algebra.

There is an important type of Lie bialgebras called coboundary
bialgebras. Namely, let $L$ be a Lie algebra and
$r=\sum\limits_ia_i\otimes b_i$ from $(id-\tau)(L\otimes L)$, that
is, $\tau(r)=-r$. Define a comultiplication $\Delta_r$ on $L$ by
$$
\Delta_r(a)=\sum\limits_i[a_i,a]\otimes b_i-a_i\otimes[a,b_i]
$$
for all $a\in L$. It is easy to see that $\Delta_r$ is a 1-cocycle.
In \cite{BD} it was proved that $(L,\Delta_r)$ is a Lie coalgebra if
and only if the element
$$
C_L(r)=[r_{12},r_{13}]-[r_{23},r_{12}]+[r_{13},r_{23}]
$$
is $L$-invariant. Here
$[r_{12},r_{13}]=\sum\limits_{ij}[a_i,a_j]\otimes b_i\otimes b_j$,
$[r_{23},r_{12}]=\sum\limits_{ij}a_i\otimes[a_j,b_i]\otimes b_j$,
 and $[r_{13},r_{23}]=\sum\limits_{ij} a_i\otimes a_j\otimes [b_i,b_j]$.
In particular, if
\begin{equation}\label{lieYB}
\sum\limits_{ij}[a_i,a_j]\otimes b_i\otimes
b_j-a_i\otimes[a_j,b_i]\otimes b_j+a_i\otimes a_j\otimes
[b_i,b_j]=0,
\end{equation}
then the pair $(L,\Delta_r)$ is a Lie bialgebra. The equation
\eqref{lieYB} is called \emph{the classical Yang-Baxter equation}.

The equation \eqref{lieYB} can de considered for every variety of
algebras. Let $A$ be an arbitrary algebra and
$r=\sum\limits_ia_i\otimes b_i\in A\otimes A$. Then the equation
\begin{equation}\label{YB}
  C_A(r)=r_{12}r_{13}+r_{13}r_{23}-r_{23}r_{12}=0
\end{equation}
is called the classical Yang-Baxter equation on $A$. Here the
subscripts specify the way of embedding $A\otimes A$ into $A\otimes
A\otimes A$, that is, $r_{12}=\sum\limits_{i} a_i\otimes b_i\otimes
1$, $r_{13}=\sum_i a_i\otimes 1\otimes b_i$, $r_{23}=\sum_i 1\otimes
a_i\otimes b_i$. Note that $C_A(r)$ is well defined even if $A$ is
non-unital. The equation for different varieties of algebras was
considered in \cite{Zhelyabin98, Aquiar, Zhelyabin, Polishchuk,
Gme}.

An element $r=\sum\limits_{i}a_i\otimes b_i\in A\otimes A$ induces a
comultiplication $\Delta_r$ on $A$:
$$
\Delta_r(a)=\sum\limits a_ia\otimes b_i-a_i\otimes ab_i
$$
for all $a\in A$.

In \cite{Zhelyabin97} the following definition of bialgebra in sense
of Drinfeld for any variety of algebras was given:

{\bf Definition.} Let $\mathcal{M}$ be an arbitrary variety of
algebras and let $A$ be an algebra from $\mathcal{M}$ with a
comultiplication $\Delta$. The pair $(A,\Delta)$ is called \emph{an
$\mathcal{M}$-bialgebra (in the sense of Drinfeld)\/} if its
Drinfeld double $D(A)$ belongs to $\mathcal{M}$.

For Jordan, associative, alternative and Malcev bialgebras it is
knows that if $r$ is a skew-symmetric solution of the classical
Yang-Baxter equation \eqref{YB} on $A$, then $(A,\Delta_r)$ is a
bialgebra of corresponding variety.

In \cite{Zhelyabin} it was proved that if $A$ is an arbitrary simple
algebra equipped with a comultiplication $\Delta$ and $I$ is a
proper ideal of $D(A)$, then $dim(I)=dim(A)$.

If $L$ is a simple finite-dimensional Lie algebra over the field of
complex numbers, then the Drinfild double $D(L)$ is either contain a
radical $R$ such that $D(L)=L\oplus R$ (semidirect sum) and $R^2=0$
or  $D(L)=L_1\oplus L_2$
--- direct sum of two simple ideals that are isomorphic to $L$
\cite{stolin}. In the same paper bialgebra structures on $L$ in both
cases were described.

If $(L,\Delta)$ is a Lie bialgebra and the radical $R$ of $D(L)$ is
not zero then the comultiplication is induced by a skew-symmetric
solution of the classical Yang-Baxter equation, that is
$\Delta=\Delta_r$ for some $r=\sum a_i\otimes b_i\in
(id-\tau)(L\otimes L)$ such that $C_L(r)=0$ (see \cite{stolin} for
the case when $F$ is the field of complex numbers or \cite{GMM} for
the more general case when $L$ is a Malcev algebra). Define an
operator $R: L\rightarrow L$ by:
\begin{equation}\label{RB}
R(x)=\sum\limits \langle a_i,x \rangle b_i,
\end{equation}
where $\langle\cdot,\cdot\rangle$ is the Killing form on $L$. It is
well known that $R$ is a Rota-Baxter operator  of weight 0 on $L$
(see \cite{BD}-\cite{STS}).

\section{Bialgebras with semisimple Drinfeld double and Rota-Baxter  operators of non-zero weights.}

In this section  $A$ is an arbitrary simple finite-dimensional
 algebra over a
field of characteristic zero equipped with a comultiplication
$\Delta$ such that $(A,\Delta)$ is a bialgebra such that $D(A)$ is a
direct sum of simple ideals.

Since $dim(I)=dim(A)$ for every proper ideal $I$ of $D(A)$ we have
that $D(A)=A_1\oplus A_2$ where $A_i$ are simple algebras and $dim
A_i=dim A$ \cite{Zhelyabin}.
\begin{proposition}\label{p1}
There are two linear mapping $\phi_i:A^*\rightarrow A$, $i=1,2$ such
that for all $f\in A^*:\ f-\phi_i(f)\in A_i$. Moreover,
$A_i=\{f-\phi_i(f)|\ f\in A^*\}$.
\end{proposition}
\begin{proof}
Since $A$ is simple then intersections $A$ with $A_i$ are equal to
zero. It means that $D(A)=A\oplus A_1=A\oplus A_2$ (as vector
spaces) and for every $f\in A^*$ there are unique elements $a_i\in
A(i=1,2)$ such that $f+a_i\in A_i$. For $i=1,2$ define $\phi_i$ as
$$
\phi_i(f)=-a_i
$$
if $f+a_i\in A_i$. Since $dimA_i=dim A^*$ we have that
$A_i=\{f-\phi_i(f)|\ f\in A^*\}$.

\end{proof}

Consider $a\in A$. There are unique $l_i\in L_i(i=1,2)$ such that
$a=l_1+l_2$. By proposition \ref{p1} $l_i=f_i+\phi_i(f_i)(i=1,2)$
and $a=(f_1-\phi_1(f_1))+(f_2-\phi_2(f_2))$. Thus, $f_1=-f_2$ and we
proved that for every $a\in A$ there is $f\in A^*$ such that:
\begin{equation}\label{eq2}
a=-\phi_1(f)+\phi_2(f).
\end{equation}

If for some $a\in A$ there are two elements $f_1\in A^*$ and $f_2\in
A^*$ such that $a=-\phi_1(f_i)+\phi_2(f_i)$ then we get
$(f_1-\phi_1(f_1))-(f_1-\phi_2(f_1))=(f_2-\phi_1(f_2))-(f_2-\phi_2(f_2))$.
Since $f-\phi_i(f)\in A_i$ we obtain that $f_1=f_2$

Define a map $\psi:L\rightarrow A^*$ as $\psi(a)=f$ if
$a=-\phi_1(f)+\phi_2(f)$. It is easy to see that $\psi$ is an
isomorphism of vector spaces.

We will need the following properties of the maps $\phi_i$ and
$\psi$:
\begin{proposition}\label{p2}
1. For all $f,g\in A^*$ and $i=1,2$:
\begin{equation}\label{eq3}
\phi_i(fg)=\phi_i(f)\phi_i(g).
\end{equation}

2. For all $a,b\in A$:
\begin{equation}\label{eq4}
\psi(ab)=\psi(a)\leftharpoondown b=a\rightharpoondown \psi(b).
\end{equation}

3. For all $a,b\in A$:
\begin{equation}\label{eq5}
\psi(ab)=\psi(a)\psi(b)-\phi_1(\psi(a))\rightharpoondown\psi(b)-\psi(a)\leftharpoondown\phi_1(\psi(b)).
\end{equation}
and
\begin{equation}\label{eq6}
\psi(ab)=-(\psi(a)\psi(b)-\phi_2(\psi(a))\rightharpoondown\psi(b)-\psi(a)\leftharpoondown\phi_2(\psi(b))).
\end{equation}

\end{proposition}
\begin{proof}

Take $f,g\in A^*$ and fix $i=1,2$. From the definition of $\phi_i$
we have that $f-\phi_i(f)=p\in A_i$ and $g-\phi_i(g)=q\in A_i$.
Then, since $A_i$ is an ideal of $D(A)$, $fg=\phi_i(f)\phi_i(g)+s$
for some $s\in A_i$. In means that $fg-\phi_i(f)\phi_i(g)\in A_i$
and by the definition of $\phi_i$ we conclude that
$\phi_i(fg)=\phi_i(f)\phi_i(g)$.

Let us prove \eqref{eq4}. Let $a,b\in A$. Then
$$
ab=(\psi(a)-\phi_1(\psi(a)))b-(\psi(a)-\phi_2(\psi(a)))b=
$$
$$
(\psi(a)\leftharpoondown b+\psi(a)\rightharpoonup
b-\phi_1(\psi(a))b)-(\psi(a)\leftharpoondown
b+\psi(a)\rightharpoonup b-\phi_2(\psi(a))b).
$$
Since $A_1$ is an ideal, the expression in the first brackets lies
in $A_1$. Thus, $\psi(ab)=\psi(a)\leftharpoondown b$ by the
definition of the map $\psi$. Similar arguments show that
$\psi(ab)=a\rightharpoondown \psi(b)$ and \eqref{eq4} is proved.

In order to prove \eqref{eq5} and \eqref{eq6} consider elements
$a,b\in A$. We have:
$$
a=(\psi(a)-\phi_1(\psi(a)))+(-\psi(a)+\phi_2(\psi(a)))
$$
$$
b=(\psi(b)-\phi_1(\psi(b)))+(-\psi(b)+\phi_2(\psi(b)))
$$
Multiplying $a$ and $b$ in $D(A)$ we get:
$$
ab=(\psi(a)\psi(b)-\phi_1(\psi(a))\rightharpoondown\psi(b)-\psi(a)\leftharpoondown\phi_1(\psi(b))+x_1)+
$$
$$
((\psi(a)\psi(b)-\phi_2(\psi(a))\rightharpoondown\psi(b)-\psi(a)\leftharpoondown\phi_2(\psi(b)))+x_2)
$$
where $x_i\in A$ ($i=1,2$). This proves \eqref{eq4} and \ref{eq5}.
\end{proof}

\begin{theorem}\label{t1}
Let $R:L\rightarrow A$ be  an operator defined as
\begin{equation}\label{d1}
R(a)=\phi_1(\psi(a)).
\end{equation}
 Then $R$ is a Rota-Baxter operator of weight 1 on $A$.
\end{theorem}
\begin{proof}
Using \eqref{p2}-\eqref{eq5} and \eqref{d1} for all $a,b\in A$ we
compute:
$$
R(R(a)b+aR(b)+ab)-R(a)R(b)=
$$
$$
=\phi_1(\psi(R(a)b+aR(b)+ab)-\phi_1(\psi(a))\phi_1(\psi(b))=
$$
$$
=\phi_1(R(a)\rightharpoondown \psi(b)+\psi(a)\leftharpoondown
R(b)+\psi(ab)-\psi(a)\psi(b))=0.
$$

Thus, $R$ is a Rota-Baxter operator of weight 1 on $A$.

\end{proof}

Using similar arguments one can proof the following:
\begin{theorem}
Let $Q$ be  an operator $Q:A\rightarrow A$ defined as
\begin{equation}\label{d2}
Q(a)=\phi_2(\psi(a)).
\end{equation}
 Then $Q$ is a Rota-Baxter operator of weight -1 on $A$.
\end{theorem}

\section{Rota-Baxter operators and classical Yang-Baxter equation for simple anti-commutative algebras.}

Let $L$ be a simple finite-dimensional anti-commutative algebra over
a field $F$ of characteristic zero and $r=\sum a_i\otimes b_i\in
L\otimes L$ is a solution of the classical Yang-Baxter equation
\eqref{lieYB} on $L$. Recall, that if $\Delta$ is a comultiplication
on $L$ then bialebra $(L,\Delta)$ is anti-commutative if and only if
the Drinfeld double is anti-commutative.

For convenience we will call an algebra $A$ semisimple if $A$ is a
direct sum of simple ideals.

\begin{theorem}
Let $(L,\Delta)$ be a structure of an anti-commutative bialgebra on
a simple anti-commutative algebra $L$ with semisimple non-simple
Drinfeld double and $R$
--- the Rota-Baxter operator defined as in \eqref{d1}. Then there is
a non skew-symmetric solution of the classical Yang-Baxter equation
\eqref{lieYB} $r=\sum a_i\otimes b_i$ such that $r+\tau(r)$ is
$L-invariant$ and a non-degenerate associative symmetric bilinear
form $\omega$ such that for all $a\in L$:
\begin{equation}\label{RBN}
R(a)=\sum\limits_i \omega(a_i,a)b_i.
\end{equation}
\end{theorem}
\begin{proof}
Since $L$ is simple $D(L)=L_1\oplus L_2$. Consider maps
$\phi_i:L^*\rightarrow L(i=1,2)$ defined in proposition 1. Since $L$
is a finite-dimensional algebra there is an element
$r_1=\sum\limits_i a_i\otimes b_i\in L\otimes L$ such that
$\phi_1(f)=\sum f(a_i)b_i$ for all $f\in L^*$.

Similarly, there is an element $r_2=\sum\limits_i c_i\otimes d_i\in
L\otimes L$ such that $\phi_2(f)=\sum f(c_i)d_i$. Since $\phi_1$ is
a homomorphism  for all $f,g\in L$ we have
$$
\phi_1([f,g])=-\sum [f,g](a_i)b_i=-\sum
f([a_j,a_i])g(b_j)b_i+f(a_j)g([b_j,a_i])b_i= $$ $$
=[\phi_1(f),\phi_1(g)]=\sum f(a_j)g(a_i) [b_j,b_i].
$$
Therefore $r_1$ is a solution of \eqref{lieYB}. Similar arguments
show that $r_2$ is also a solution of \eqref{lieYB}.

Since $L_1$ and $L_2$ are simple and the form $Q$ is associative,
$Q(L_1,L_2)=0$. Therefore for all $f,g\in L^*$ we have
$Q(f-\phi_1(f), g-\phi_2(g))=0$. Hence,
$$\sum\limits_i \langle f\otimes g, a_i\otimes b_i+d_i\otimes
c_i\rangle=0.$$ Consequently,
\begin{equation}\label{usl1}
r_1+\tau(r_2)=0.
\end{equation}

Also we have $(f-\phi_1(f))(g-\phi_2(g))=0$.
 Therefore,
 $$fg-f\leftharpoondown \phi_2(g)-
 \phi_1(f)\rightharpoondown g=0.$$
The last equality means that for all $a\in \mathbb{M}$
 $$
 fg(a)=\sum\limits_i f(g(c_i)d_ia)+g(f(a_i)ab_i).$$
 $$
 \langle f\otimes g, \Delta(a)\rangle=\langle f\otimes g,
\sum\limits_i d_ia\otimes c_i+a_i\otimes ab_i\rangle.
$$

Using  \eqref{usl1} we finally obtain
$$
 \langle f\otimes g, \Delta(a)\rangle=-\langle f\otimes g,
\sum\limits_i a_ia\otimes b_i-a_i\otimes ab_i\rangle.
$$
Thus, $\Delta=-\Delta_{r_1}$.

Anti-commutativity of $D(L)$ is equivalent to
$\tau(\Delta_{r_1}(a))=-\Delta_{r_1}(a)$. Therefore
$[r_1+\tau(r_1),a]=0$ for all $a\in L$ and $r_1+\tau(r_1)$ is
$L$-invariant.


Define a form $\omega(\cdot,\cdot)$ on $L$ by:

 $$\omega(a,b)=Q(\psi(a),b)
 $$
for all $a,b\in L$. It is clear that $\omega$ is bilinear and
non-degenerate. Let us prove that $\omega$ is associative and
symmetric.

Let $a,b\in L$ and  $f_1,f_2\in L^*$ such that $\psi(a)=f_1$ and
$\psi(b)=f_2$. Since $Q(L_1,L_2)=0$ we have:

$$
0=Q(f_1-\phi_2(f_1),f_2-\phi_1(f_2))=-Q(f_1,\phi_1(f_2))-Q(\phi_2(f_1),f_2).
$$

Similarly one can prove that
$$
Q(f_1,\phi_2(f_2))+Q(\phi_1(f_1),f_2)=0.
$$

Summing up the last two equations and using \eqref{eq2} we obtain
that $Q(f_1,b)-Q(a,f_2)=0$. Thus,  $\omega(a,b)=\omega(b,a)$ for all
$a,b\in L$.

Let $a,b,c\in L$. Using \eqref{eq4} and associativity of the form
$Q$ we compute:
$$
\omega([a,b],c)=Q(\psi([a,b]),c)=Q(\psi(a)\leftharpoondown
b,c)=Q([\psi(a),b],c)= $$ $$=Q(\psi(a),[b,c])=\omega(a,[b,c]).
$$

Thus, $\omega$ is a bilinear non-degenerate symmetric associative
form on $L$. And for all $a\in L$ we have
$$
R(a)=\phi_1(\psi(a))=\sum Q(\psi(a),a_i)b_i=\sum \omega(a,a_i)b_i.
$$
\end{proof}

\begin{corollary}
Let $(L,\Delta)$ be a structure of an anti-commutative bialgebra on
simple Lie algebra $L$ with semisimple non-simple Drinfeld double
and $R$ --- the Rota-Baxter operator defined as in \eqref{d1}. Then
there is a non skew-symmetric solution of the classical Yang-Baxter
equation \eqref{lieYB} $r=\sum a_i\otimes b_i$ such that $r+\tau(r)$
is $L-invariant$ and for all $a\in L$:
$$
R(a)=\sum\limits_i \langle a_i,a\rangle b_i.
$$
Here $\langle\cdot,\cdot \rangle$ is the Killing form on $L$.
\end{corollary}
\begin{proof}
By theorem 3 there is a non-skew symmetric solution of CYBE $r_1$
with $L$-invariant symmetric part and $R(a)=\sum \omega(a,a_i)b_i$
for some non-degenerate associative symmetric bilinear form $\omega$
on $L$. But since $L$ is simple, there is a non-zero $\lambda\in F$
such that $\omega(a,b)=\lambda\langle a,b\rangle$. Therefore
$$R(a)=\sum \omega(a,a_i)b_i=\lambda \sum \langle a, a_i\rangle b_i.
$$

It is left to define $r$ as $r=\frac{1}{\lambda}r_1$ to prove the
theorem.
\end{proof}

\begin{theorem}
Let $L$ be a simple anti-commutative algebra and
$r=\sum\limits_ia_i\otimes b_i$ is a non skew-symmetric solution of
CYBE \eqref{lieYB} such that $r+\tau(r)$ is $L$ invariant. Then
there is a non-degenerate symmetric associative bylinear form
$\omega$ on $L$ such that an operator $R:L\rightarrow L$ defined as
$$
R(a)=\sum\limits_i \omega(a_i,a) b_i
$$
is a Rota-Baxter operator of non-zero weight.
\end{theorem}
\begin{proof}

Define a comultiplication $\Delta_r$ on  $L$ as
$$
\Delta_r(a)=[r,a]
$$
for all $a\in L$. Since $r+\tau(r)$ is $L$ invariant we have that
$[r,a]=-[\tau(r),a]$ for all $a\in L$. Therefore
$\tau(\Delta(a))=-\Delta(a)$ and $D(L)$ is anti-commutative.

 We want to prove that Drinfeld double of the
bialgebra is semisimple and non-simple. For this consider a map
$\phi_1:L^*\rightarrow L$ defined as
\begin{equation}\label{m1}
\phi_1(f)=-\sum\limits_i f(a_i)b_i
\end{equation}

Since $r$ is a solution of \eqref{lieYB}, $\phi_1$ is a
homomorphism.

 Consider a subspace $L_1=\{ f+\phi_1(f)|\ f\in L^*\}$. We want to
to prove that $L_1$ is an ideal of $D(L)$. For every $a\in L$ we
have:
$$
[f+\phi_1(f),a]=f\leftharpoondown a+f\rightharpoonup
a+[\phi_1(f),a]=$$
$$=f\leftharpoondown a +\sum f([a_i,a])b_i+\sum
f(a_i)[b_i,a]-f(a_i)[b_i,a]=f\leftharpoondown a +\sum f([a_i,a])b_i
$$
On the other hand,
$$
\phi_1(f\leftharpoondown a)=-\sum (f\leftharpoondown
a)(a_i)b_i=-\sum f([a,a_i])b_i=\sum f([a_i,a])b_i$$
 and
$[f+\phi_1(f),a]\in L_1. $

 Now take $g\in L^*$.
We have
$$
[f+\phi_1(f),g]=fg+\phi_1(f)\rightharpoondown
g+\phi_1(f)\leftharpoonup g
$$
Since $\phi_1$ is a homomorphism we have:
$$
\phi_1([f,g]+\phi_1(f)\rightharpoondown
g)=[\phi_1(f),\phi_1(g)]+\phi_1((f)\rightharpoondown g)=$$ $$=\sum
f(a_i)g(a_j)[b_i,b_j]-f(a_i)g([a_j,b_i])b_j= \phi_1(f)\leftharpoonup
g$$ and $L_1$ is a proper ideal of $D(L)$. By definition $dim
(L_1)=dim(L)$.

If $V\subset D(L)$ then by $V^{\perp}$ denote the complement of $V$
with respect to $Q$, that is $V^{\perp}=\{l\in D(L)|\ Q(l,V)=0\}$.
If $V$ is a proper subalgebra in $D(L)$ then $V^{\perp}$ is a proper
ideal of $D(L)$.

 Consider $L_1^{\perp}$. Since $L_1$ is an  ideal of $D(L)$
 then $L_1^{\perp}$ is also an ideal if $D(L)$.  It means that $I=L_1\cap L_1^{\perp}$ is a proper ideal
 of $D(L)$ and therefore $I=L_1$  or $I=0$.

 Since $r+\tau(r)\neq 0$ there is $h=\sum f_j\otimes
g_j\in L^*\otimes L^*$ such that $\sum\limits_{i,j}
f_j(a_i)g_j(b_i)+f_j(b_i)g_j(a_i)\neq 0$. Then $\sum
Q(f_j+\phi_1(f_j),g_j+\phi_1(g_j))\neq0$ and $I=L_1\cap
L_1^{\perp}=0$. Therefore $D(L)=L_1\oplus L_1^{\perp}$ and
$L_1^{\perp}$ is isomorphic to the quotient algebra $D(L)/L_1$.

On the other hand, $D(L)=L\oplus L_1$ (as vector spaces), therefore
$L$ is also isomorphic to the quotient algebra $D(L)/L_1$. Thus, $L$
and $L_1^{\perp}$ are isomorphic. Similar arguments show that $L$ is
isomorphic to $L_1$ and  $D(L)$ is a sum of two simple ideals. Thus,
$D(L)$ is a semisimple non-simple algebra. And the statement of the
theorem follows from the definition of maps $\phi_i$ and theorems 1
and 3.

\end{proof}

\begin{corollary}

Let $L$ be a simple Lie algebra and $r=\sum\limits_ia_i\otimes b_i$
is a non skew-symmetric solution of CYBE \eqref{lieYB} such that
$r+\tau(r)$ is $L$ invariant. Then an operator $R:L\rightarrow L$
defined as
$$
R(a)=\sum\limits_i \langle a_i,a\rangle b_i
$$
is a Rota-Baxter operator of non-zero weight. Here
$\langle\cdot,\cdot \rangle$ is the Killing form on $L$.
\end{corollary}

{\bf Remark.} If $r$ is a non skew-symmetric solution of the
classical Yang-Baxter equation such that $\tau(r)+r$ is
$L$-invariant, then so is an element $r_1=\tau(r)$. Elements $r$ and
$-r_1$ induce the same bialgebra structure on $L$. And if $Q$ is a
Rota-Baxter operator with respect to $r_1$ then $Q$ has weight
$\lambda$ and $R+Q=-\lambda id$, where $\lambda$ is a weight of
operator $R$ and $id$ is the identity operator on $L$. This is not
surprising since if $R$ is a Rota-Baxter operator of weight 1, then
so is $-id-R$.

{\bf Example 1.}

Let $L=sl_2$, $x,h,y$ is the standard basis of $L$, that is $hx=2x$,
$hy=-2y$, $xy=h$.

Consider an element
$$
r=\alpha(h\otimes x-x\otimes h)+\frac{1}{4}h\otimes h+x\otimes y,
$$
where $\alpha\in F$. Then for every $\alpha\in F$ $r$ is a solution
of the classical Yang-Baxter equation and therefore induces on $L$ a
structure of a Lie bialgebra with semisimple Drinfeld double. By
theorem 4 the operator $R$ defined as \eqref{RBN} is a Rota-Baxter
operator of non-zero weight.

We have:
$$
R(x)=0,\ R(h)=2h+8\alpha x,\ \ R(y)=4(y-\alpha h).
$$

Direct computations shows that $R$ is a Rota-Baxter operator of
weight -4.

In order to compute the second Rota-Baxter operator $Q$ we need to
consider
$$
r_1=\tau(r)=-\alpha(h\otimes x-x\otimes h)+\frac{1}{4}h\otimes
h+y\otimes x.
$$

Thus,
$$
Q(x)=4x,\ Q(h)=-8\alpha x+2h,\ Q(y)=4\alpha h
$$
and $Q$ is a Rota-Baxter operator of weight -4. Note that $Q+R=4id$
where $id$ is the identity operator on $L$.

 The following example
shows that in general a non skew-symmetric solutions of the CYBE not
necessary induces a Rota-Baxter operator of non-zero weight.

 {\bf Example 2.} Let
$L=sl_2(\mathbb{C})$ and let $x,h,y$ be the standard basis of $L$,
that is $[h,x]=2x,\ [h,y]=-2y,\ [x,y]=h$. Consider an element
$r=x\otimes x$. Obviously, $r$ is a non skew-symmetric solution of
\eqref{lieYB}. The corresponding operator $R$ acts as follow:
$$
R(x)=0,\ R(h)=0,\ R(y)=4x$$ and is a Rota-Baxter operator of zero
weight.

\section{Rota-Baxter operators on the split simple Malcev algebra.}
In this section we consider simple non-Lie Malcev algebra.

Malcev algebras were introduced by A.I. Malcev \cite{M55} as tangent
algebras for local analytic Moufang loops.  The class of Malcev
algebras generalizes the class of Lie algebras and has a well
developed theory \cite{KuzSh}.

An important example of a non-Lie Malcev algebra is the vector space
of zero trace elements of a Caley-Dickson algebra with the
commutator bracket multiplication \cite{S62,K68}. In \cite{versh}
some properties of Malcev bialgebras were studied. In particular,
there were found  conditions for a Malcev algebra with a
comultiplication to be a Malcev bialgebra. In \cite{GMM} it was
found a connection between solutions of the classical Yang-Baxter
equation on Malcev algebras with Malcev bialgebras.

\begin{definition}An anticommutative algebra is called a Malcev algebra if for
all $x,y,z\in M$ the following equation holds:
\begin{equation}\label{mal1}
J(x,y,xz)=J(x,y,z)x,
\end{equation}
where  $J(x,y,z)=(xy)z+(yz)x+(zx)y$ is the jacobian of elements
$x,y,z$.
\end{definition}

In \cite{GMM} it was proved that non skew-symmetric solutions of the
classical Yang-Baxter equation on simple Malcev algebra $M$ induces
on $M$ a structure of Malcev biialgebra with semisimple Drinfeld
double.

Let $M$ is a simple Malcev algebra. Then $M$ is a simple Lie algebra
or the 7-dimentional Malcev algebra isomorphic to the commutator
algebra of traceless elements of the split Caley-Dixon algebra
\cite{Kuz}.

{\bf Example 3}. Let $\mathbb{M}$ be the simple Malcev algebra over
the field of complex numbers $\mathbb{C}$. In this case $\mathbb{M}$
has a basis
 $h,x,x',y,y',z,z'$
 with the following
 table of multiplication:
$$
hx=2x,\ hy=2y,\ hz=2z,$$
$$
hx'=-2x',\ hy'=-2y',\ hz'=-2z',$$
$$
xx'=yy'=zz'=h,
$$
$$
xy=2z',\ yz=2x',\ zx=2y',$$
$$x'y'=-2z,\ y'z'=-2x,\
z'x'=-2y.$$ The remaining products are zero. In \cite{GMM} it was
proved that up to automorphism any non skew-symmetric solution of
the classical Yang-Baxter equation $r$ has the following form:
$$
r=r_0+\frac{1}{4}h\otimes h+x\otimes x'+y'\otimes y+z\otimes z'
$$
where
$$
r_0=\alpha (h\otimes x-x\otimes h)+ \beta(h\otimes y'-y'\otimes
h)+\gamma(h\otimes z-z\otimes h) +$$
$$
\delta(x\otimes y'-y'\otimes x)- 2\beta(x\otimes z-z\otimes
x)+\mu(y'\otimes z-z\otimes y').
$$
Scalars $\alpha,\beta,\gamma,\delta,\mu$  are arbitrary. The
corresponding normalized(of weight -1) Rota-Baxter  operator is the
following:
$$
R(h)=\dfrac{1}{2}h+2\alpha x+2\beta y'+2\gamma z,\ R(x)=0,\
R(x')=x'-\alpha h+\delta y'-2\beta z
$$
$$
R(y)=y-\beta h-\delta x+\mu z,\ R(y')=R(z)=0,\ R(z')=z'-\gamma
h+2\beta x-\mu y'.
$$


\bigskip
\begin{thebibliography}{1}

\bibitem{Br} Baxter G. An analytic problem whose solution follows from a simple algebraic identity
// Pacific J. Math. 1960. Vol. 10. p. 731–742.

\bibitem{Atk} Atkinson, F.V.: Some aspects of Baxter’s functional equation. J. Math. Anal. Appl. 7,
1–30 (1963)

\bibitem{Rota}  Rota G.-C. Baxter algebras and combinatorial identities I and II // Bull. Amer. Math. Soc.
1969. Vol. 75. P. 325–334.

\bibitem{Miller} J.B. Miller  Some properties of Baxter operators // Acta Math. Acad. Sci. Hungar.
1966. Vol. 17. P. 387–400.

\bibitem{Car} Cartier P. On the structure of free Baxter algebras // Adv. Math. 1972. Vol. 9. P.
253–265.

\bibitem{Guo} Guo L. An Introduction to Rota—Baxter Algebra. Surveys of Modern Mathematics.
Vol. 4. Somerville, MA: International Press; Beijing: Higher
education press, 2012. 226 p.

\bibitem {Drinf}
Drinfeld V.G. Hamiltonian structures on Lie groups, Lie bialgebras
and the geometric meaning of the classical Yang-Baxter equation,
\emph{Sov, Math, Dokl}, 27 (1983), 68-71.

\bibitem{stolin} Stolin A.A. Some remarks on Lie bialgebra structures on
simple complex Lie algebras,  \emph{Comm. in Algebra}, 27, 9(1999)
4289-4302

\bibitem{BD} Belavin A.A., Drinfeld V.G., Solutions of the classical Yang - Baxter equation for simple Lie algebras, \emph{Funct. Anal.
Appl.}, 16(3) (1982),  159–180.

\bibitem{STS} Semenov-Tyan-Shanskii M.A. What is a classical r-matrix? // Funct. Anal. Appl.
1983. Vol. 17, N 4. P. 259–272.

\bibitem{Zhelyabin97}
Zhelyabin V.N., Jordan bialgebras and their relation to Lie
bialgebras, \emph{Algebra and logic}, vol. 36, 1 (1997), 1-15.

\bibitem {Zhelyabin98}
Zhelyabin V.N., Jordan bialgebras of symmetric elements and Lie
bialgebras, \emph{Siberian mathematical journal}, vol. 39, 2 (1998),
261--276.


\bibitem {Zhelyabin}
Zhelyabin V.N., On a class of Jourdan D-bialgebras,\emph{ St.
Petersburg Mathematical Journal}, 2000, 11:4, 589–609.

\bibitem{Gme} Goncharov M.E., The classical Yang-Baxter equation on alternative algebras:
The alternative D-bialgebra structure on Cayley-Dickson matrix
algebras,  \emph{Siberian mathematical journal}, vol. 48, 5 (2007)
809-823.

\bibitem {ANQ}
Anquela J.A., Cortes T. Montaner F. Nonassociative Coalgebras//
Comm.Algebra. 1994. V. 22, N 12. P. 4693--4716.

\bibitem {GMM} Goncharov M.E. Structures of Malcev Bialgebras on a Simple Non-Lie Malcev Algebra. // Communications in algebra, 40, 8, (2012) 3071-3094.

\bibitem {DR} Drinfeld V.G., Quantum groups, in: Proc. Internat. Congr. Math., Berkeley, 1986
(ed. A. M. GLEASON), Amer. Math. Soc., Providence, RI, 1987, pp.
798-820.

\bibitem{M55}  Malcev A.I., Analytic loops, \emph{Matem. Sb.}, 36(78):3 (1955), 569–576 (in Russian).

\bibitem{KuzSh} Kuzmin E. N., Shestakov I. P., Nonassociative structures,
\emph{Algebra – 6, Itogi Nauki i Tekhniki. Ser. Sovrem. Probl. Mat.
Fund. Napr.}, 57, VINITI, Moscow, 1990, 179–266 (in Russian).

\bibitem{S62} Sagle A.A., Simple  Malcev algebras over fields of
characteristic zero, \emph{Pacific J. Math.} 12(1962), 1047-1078.

\bibitem{K68} Kuzmin E.N., Maltsev algebras and their representations, \emph{Algebra and logic}, vol. 7, 4 (1968) 233-244.

\bibitem{versh} Vershinin V.V., On Poisson-Malcev Structures, \emph{Acta Applicandae
Mathematicae}, 75(2003)  281-292.

\bibitem {GMM} Goncharov M.E. Structures of Malcev Bialgebras on a Simple Non-Lie Malcev Algebra. // Communications in algebra, 40, 8, (2012) 3071-3094.

\bibitem {Aquiar}
Aguiar M. On the associative analog of Lie bialgebras, \emph{Journal
of Algebra}, 244,(2001), 492--532.

\bibitem{Polishchuk}  Polishchuk A. Clasic Yang~--- Baxter Equation and the
A-constraint, \emph{Advances in Mathematics}, vol. 168, No. 1, 2002,
56-96.

\bibitem{Kuz} Kuzmin E. N., Structure and representations of finite-dimensional simple Malcev
algebras, \emph{Issled. po teor. kolec i algebr (trud. inst. matem.
SO RAN SSSR,16)}, Novosibirsk, Nauka, 1989, 75-101 (in Russian).

\end{thebibliography}
 \end{document}